\documentclass[11pt]{amsart}
\usepackage{color}
\usepackage{pstool}
\usepackage{tikz-cd}

\usepackage[latin1]{inputenc} 
\usepackage[english]{babel}
\usepackage{psfrag}
\usepackage[T1]{fontenc}
 \usepackage{lmodern}
\usepackage{mathrsfs}
\usepackage{graphicx,epsfig,amscd,graphics,xypic, subcaption}
\usepackage{amsmath,amssymb}

\usepackage{rotating}

\usepackage[margin=4cm]{geometry}

\newcommand{\C}{\mathbb{C}}
\newcommand{\Z}{\mathbb{Z}}
\newcommand{\R}{\mathbb{R}}

\newcommand{\SL}{ \mathrm{SL}(2,\mathbb{R})}

\newcommand{\M}{\mathcal{MD}_{g,n}}
\newcommand{\T}{\mathcal{TD}_{g,n}}
\newcommand{\MT}{\mathcal{MD}(\lambda)}
\newcommand{\TT}{\mathcal{TD}(\lambda)}

\newtheorem{theorem}{Theorem}

\newtheorem*{theorem*}{Theorem}   
\newtheorem*{lemma*}{Lemma}   
  
\newtheorem*{conjecture}{Conjecture}  
\newtheorem{proposition}[theorem]{Proposition}
\newtheorem{defi}{Definition}           
\newtheorem{question}{Question} 
\newtheorem{lemma}[theorem]{Lemma}

\newtheorem*{ex*}{Example \ref{exJ2} (continued)}

\title[Teichmüller dynamics and dilation tori]{Teichmüller dynamics, dilation tori and piecewise affine circle homeomorphisms}

\author[Selim Ghazouani]{Selim Ghazouani}
\address{Mathematics Institute, University of Warwick, Coventry CV4 7AL, U.K. }
\email{s.ghazouani@warwick.ac.uk}

\begin{document}
\maketitle

\begin{abstract}

We study the coarse geometry of the moduli space of dilation tori with two singularities and the dynamical properties of the action of the Teichmüller flow on this moduli space. This leads to a proof that the vertical foliation of a dilation torus is almost always Morse-Smale. As a corollary, we get that the generic piecewise affine circle homeomorphism with two break points -with respect to the Lebesgue measure- is Morse-Smale.

\end{abstract}

\section{Introduction}

This article is concerned with generic dynamical properties in families of piecewise affine circle homeomorphisms and closely related transversally affine foliations of tori. 

\vspace{3mm}

\noindent There are two competing notions of genericity in dynamical systems, the \textit{topological} one and the \textit{probabilistic} one. For a given finite-dimensional family of dynamical systems, one says a given property is

\begin{itemize}
\item \textit{topologically generic} if it is satisfied by a dense $G_{\delta}$ of the family;
\item \textit{probabilistically generic} or \textit{generic in measure} if it is satisfied by a set of parameters whose complement has Lebesgue measure equal to zero. 
\end{itemize}

\noindent The second notion is stronger than the first as it usually implies it. It is natural when given a family of dynamical systems to ask what dynamical behaviour is to be observed generically, in both existing sense of the term. The general setting in which we want to ask this question is the one of piecewise continuous bijections of the interval (including circle homeomorphisms) and foliations on surfaces (these two are to be thought of as the two sides of the same coin).  

\vspace{3mm} \noindent Foliations on surfaces satisfy the following trichotomy: it is either\footnote{We have voluntarily ignored the case of foliations having saddle connections as they are easily shown to form a negligible subset of the set of foliations.}

\begin{itemize}

\item \textit{Morse-Smale}, meaning roughly that every regular leaf accumulates to a closed attracting leave;

\item minimal and is a measured foliation;

\item or it has a closed invariant set which is locally the product of a Cantor set with an interval.

\end{itemize} We refer to the series of articles \cite{Levitt,Levitt1,Levitt2} where Levitt gives a nice classification of foliations on surfaces.

\noindent In his seminal work \cite{Peixoto,Peixoto1}, Peixoto proved that Morse-Smale foliations form a dense open subset of the (infinite dimensional) set of smooth foliations on a closed oriented surface of genus $g$ for any $g \in \mathbb{N}$, settling the question of the topological genericity for flows on surfaces. We would also like to mention \cite{Liousse} where similar results for transversally affine foliations are proven.

\vspace{2mm} \noindent  The set of smooth foliations on a given smooth surface (alternatively of smooth generalised interval exchange maps of the $S^1$) is an infinite dimensional space and it is therefore not clear what would be a good measure to put on it in order to create a framework to study probabilistic genericity. A way to bypass this difficulty is to ask the question for finite dimensional families of such maps and use the Lebesgue measure in the parameters. It is in this spirit that Arnold introduced the following family of circle diffeomorphisms (which is usually referred to as 'Arnold's tongues')

$$r_{\alpha,\epsilon} :  x \mapsto x + \alpha + \epsilon\sin(2\pi x) $$ and proved that the set of parameters $(\alpha,\epsilon)$ for which $r_{\alpha,\epsilon}$ is minimal is of positive measure - in spite of his complement containing a dense open set (made of Morse-Smale circle diffeomorphisms). This result was later generalised by Herman who went on to prove the following

\begin{theorem}[Herman, \cite{Herman}]
Let $(f_t)_{t \in [0,1]}$ be a smooth family of $\mathcal{C}^3$ circle diffeomorphisms. Assume that the rotation number of $f_0$ is different from that of $f_1$. Then the set 

$$  \big\{ t \in [0,1 \ | \ f_t \ \text{is minimal} \big\}   $$ has positive measure. 
\end{theorem} \noindent It is a remarkable feature of families of (sufficiently regular) circle diffeomorphisms that the notion of topological and probabilistic genericity do not agree. An equivalent statement for foliations would be the following: in one-parameter families of smooth foliations on the torus, there is a set of parameters of positive measure for which the foliation is minimal.

\vspace{3mm} \noindent We put forward in this article an investigation of the case of piecewise affine homeomorphisms of the circle and equivalently of transversally affine foliations on the torus. In his celebrated article \cite{Herman1} Herman opens a discussion on the family of piecewise affine circle homeomorphisms with two discontinuities for the derivative. This family can easily be parametrised by $4$ real parameters (the discontinuity points of the derivative and their images characterising completely such a piecewise affine map of the circle) and is therefore endowed with the standard Lebesgue measure. Our main theorem is the following

\begin{theorem}
\label{circleaffine}
Almost every piecewise affine circle homeomorphism with two break points is dynamically Morse-Smale.
\end{theorem}

\noindent This theorem is a corollary of a more general study of geometric objects which we call \textit{dilation surfaces}. These are singular structures on surfaces modelled on $\C$ through the group of dilations (maps of the form $z \mapsto az + b$ with $a$ real and positive). They are typically the surfaces one gets when gluing parallel sides of a Euclidean polygon along dilations, see Figure \ref{example} below.

\begin{figure}[!h]
  \centering
  \includegraphics[scale=0.4]{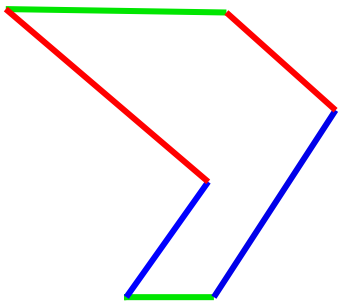}
  \caption{A hexagonal model whose sides of same colour are glued together to get a dilation torus.}
  \label{example}
\end{figure}

\noindent These surfaces come with directional foliations whose first return maps are piecewise affine. For a rigorous introduction to this material we refer to Section \ref{dilation}.  This change of viewpoint, analogous in many ways to the linear interval exchange transformations/translation surfaces duality, is very fruitful in the sense that the geometric properties of their moduli spaces encode the way the geometry of these surfaces can degenerate; these possible degenerations relating in a subtle way to the dynamical properties of their directional foliations by means of a renormalisation procedure acting upon the moduli space of dilation surfaces called the \textit{Teichmüller flow}.
\noindent We believe this article draws a clear analogy between the action of this flow and the action of the geodesic flow on hyperbolic manifolds of infinite volume. As we will see 

\begin{itemize}

\item Morse-Smale foliations corresponds to orbits of the Teichmüller flow escaping in a 'funnel' (non-compact part of the moduli space of infinite volume);

\item minimal foliations corresponds to recurrent orbits of the Teichmüller flow;

\item totally periodic foliations corresponds to orbits of the Teichmüller flow escaping in a 'cusp' (non-compact part of the moduli space of finite volume).

\end{itemize}

\noindent In this setting Theorem \ref{circleaffine} becomes the direct analogue of Alfohrs theorem stating that the limit set of a (sufficiently well-behaved) Kleinian group of infinite covolume has measure zero.

\vspace{3mm}

\paragraph*{\bf Outline of the article}

In Section \ref{dilation} and \ref{moduli} we gather elementary material on dilation surfaces and their moduli spaces $\MT$ (the $\lambda$ in the notation simply corresponds to fixing the singularity type). Notably we introduce the notion of \textit{dilation cylinder} which is the affine structure on a cylinder one gets by making the quotient of an angular sector of $\C$ by a dilation centred at the vertex of this angular sector. For any dilation surface $\Sigma$, we define $\Theta(\Sigma)$ to be the angle of the largest cylinder contained in $\Sigma$; $\Theta$ defines a continuous function on the moduli space of dilation surfaces. This function is going to be crucial in the description of the coarse geometry of $\MT$. 

\noindent Another key ingredient that we introduce is an action of $\SL$ on $\MT$. The \textit{Teichmüller flow} is the restriction of this action to the subgroup of diagonal matrices. We show the existence of an invariant measure $\mu$ for this $\SL$-action which is in the same class as the Lebesgue measure. The existence of such a measure is crucial to the proof of our main result.
 
\noindent In Section \ref{twosingularities} and \ref{geometry} we study in detail the geometric properties of dilation tori with two singularities. We prove a battery of technical results that allows us to get a good understanding of the way the geometry of such a dilation torus can degenerate.   

\noindent In Section \ref{volume} we give a proof that subsets of $\MT$ on which the function $\Theta$ does not exceed $\frac{\pi}{4}$ have finite $\mu$-volume. This is in a sense the technical heart of the proof as everything follows quite smoothly from this point. This calculation relies on the work performed in Section \ref{geometry}.

\noindent Finally in Section \ref{dynamics} we give a proof of the fact that minimal foliations  (thought of as points of $\MT$) form a subset of measure zero of $\MT$, see Theorem \ref{maintheorem}. The proof relies on a simple criterion for a minimal foliation in $\MT$ to be a density point of the set of minimal foliations. This criterion is given in terms of the evolution of the function $\Theta$ along the Teichmüller flow. 
\section{Dilation surfaces}
\label{dilation}

\noindent We begin with a formal definition of the notion of dilation surface. 

\begin{defi}
A dilation surface is a geometric structure on a surface modelled on the complex plane $\mathbb{C}$ through the subgroup of complex affine transformations $\R^*_+ \ltimes \C$ whose elements have real positive linear part and with a finite number of singular points of dilation-conical type.
\end{defi}

\noindent A singular point of dilation-conical type is a slight variation on usual Euclidean cone singularities of angle $2k \pi$: there is (possibly) an additional dilation factor when one computes the parallel transport around the singular point. If $\Sigma$ is a dilation surface, we denote by $S(\Sigma)$ or simply $S$ the set of its singular points.

\vspace{2mm} \noindent  A typical example of dilation surface is the structure one gets when identifying pairs of parallel sides of a polygon in the plane. Translation surfaces are particular examples of dilation surfaces. 
\vspace{2mm}

\paragraph{\bf Directional foliations} The foliations by straight lines of fixed slope of the complex plane $\C$ are preserved by the subgroup $\R^*_+ \ltimes \C$. This allows to define on every dilation surface a family of \textit{directional foliations} which are simply the trace of the aforementioned foliations. Each of these individual foliations comes with an extra \textit{transversally affine structure} (see \cite{Liousse} for a detailed discussion of this notion).

\vspace{2mm} \noindent We say that a foliation is \textit{Morse-Smale} when there exists a finite number of closed repelling or attracting leaves for which the following holds: the $\alpha$-limit and $\omega$-limit of every regular leaf is one of these closed leaves.

\subsection{Dilation cylinders}

Let $\mathcal{S}_{\theta}$ be an angular sector in $\C$ of angle $\theta \in ]0,2\pi[$ based at $0 \in \C$. We denote by $\mathcal{C}_{\theta, \lambda}$ the quotient of  $\mathcal{S}_{\theta}$ by the action of $z \mapsto \lambda z$ where $\lambda > 1$. We call this surface the \textit{dilation cylinder of angle} $\theta$ \textit{and of multiplier} $\lambda$. It is a dilation surface homeomorphic to a cylinder $S^1 \times \R$ whose boundary is totally geodesic. 

\vspace{3mm} \noindent A (dilation) cylinder of angle $\theta$ and multiplier $\lambda > 1$ in a dilation surface $\Sigma$ is a maximal affine embedding of $\mathcal{C}_{\theta, \lambda}$ in $\Sigma$. We say that a dilation surface has a cylinder in direction $\alpha \in S^1$ if the directional foliation in direction $\alpha$ has a closed leaf contain in a cylinder. 

\begin{proposition}
\label{numberofcylinders}
Let $\Sigma$ be a dilation surface of genus $g$ with $n$ singular points and let $\alpha$ be a direction in $S^1$. Then $\Sigma$ has at most $3g-3 +n$ cylinders in direction $\alpha$.
\end{proposition}

\begin{proof}
Cylinders in a given direction are pairwise disjoint and any pair of boundary curves of two different cylinders are not free-homotopic in the punctured surface. It is an elementary fact of topology of surfaces that their cannot be more than $3g-3 + n$ such disjoint non-free-homotopic cylinders. 
\end{proof}

\subsection{Triangulations and large cylinders}

In order to study geometric properties of a given dilation surface, it is often convenient to have a polygonal representation of this surface. It is natural to ask in that context what are the surfaces that can be represented as a polygon. A related question is the one of the existence of  triangulations  whose set of vertices is exactly the set of singular points and whose edges a regular geodesic segments. We call such a triangulation a \textit{geodesic triangulation}.

\vspace{2mm}

\noindent There is a natural obstruction to the existence of such a triangulation: if a dilation surface contains a cylinder of angle greater or equal to $\pi$, it cannot have a geodesic triangulation. A theorem of Veech shows that it is the only obstruction.

\begin{theorem}[Veech, \cite{VeechU}]
Let $\Sigma$ be a dilation surface of genus $g$ with $n$ singular points. Assume $2-2g-n < 0$. Then $\Sigma$ admits a geodesic triangulation if and only if $\Sigma$ does not contain any cylinder of angle larger than $\pi$.
\end{theorem} \noindent We will make use of this fact when defining moduli spaces of dilation surfaces.

\subsection{The invariant  $\Theta$ }
\label{theta}
A given dilation surface can contain an infinity of different cylinders. However it cannot contain cylinders of arbitrarily large angle. Actually we have the stronger statement:

\begin{proposition}
\label{finitecylinders}
Let $\Sigma$ be a dilation surface and let $\theta_0$ be a positive number. Then there are only finitely many cylinders in $\Sigma$ of angle larger than $\theta_0$.
\end{proposition}

\begin{proof}
This proposition is a consequence of Proposition \ref{numberofcylinders}. Indeed, assume there are infinitely many cylinders of angle larger than a given $\theta_0$. Then, by compactness of $S^1$ there is a direction $\alpha$ which is contained in infinitely many of these cylinders which contradicts Proposition \ref{numberofcylinders}.
\end{proof}

\noindent We denote by $\Theta(\Sigma)$ the angle of the largest cylinder of $\Sigma$. This number is well defined thanks to Proposition \ref{finitecylinders}. We will see later on that $\Theta$ actually defines a very useful function on the moduli space of dilation surfaces. 

\subsection{Linear holonomy}

We end this section by a short discussion about the notion holonomy. Dilation surfaces are a particular occurrence of a $(G,X)$-structure with singularities where $X$ is a model space and $G$ a group acting in a rigid way on $X$. In our case, $X = \C$ and $G= \R^*_+ \ltimes \C$. In this setting, it is possible(see \cite{Ghazouani2}) to produce an algebraic invariant called the \textit{holonomy} which a (class of) representation $ \pi_1 (\Sigma \setminus S) \longrightarrow  \R^*_+ \ltimes \C$. 

\vspace{2mm} \noindent To our purpose it is convenient to consider only the projection of this representation onto the factor $\R^*_+$ to get by this mean a representation $\rho: \pi_1(\Sigma \setminus S) \longrightarrow \R^*_+$ which factors through $\rho : H_1(\Sigma \setminus S, \Z) \longrightarrow \R^*_+$. We can therefore think of $\rho$ as an element of $H^1(\Sigma \setminus S, \R^*_+)$. We call $\rho$ the \textit{linear holonomy} of $\Sigma$. 

\vspace{2mm} \noindent This representation has the following geometric meaning. Consider a closed loop $\gamma$ on $\Sigma$. The parallel transport defined by the affine structure along $\gamma$ is a dilation by a certain factor which is exactly $\rho(\gamma)$. 

\section{Moduli spaces and action of $\mathrm{SL}(2,\mathbb{R})$}
\label{moduli}
We define in this section moduli spaces of dilation structures. For the rest of the section

\begin{itemize}
\item $g$ and $n$ be integers such that $2 - 2g - n < 0$; 

\item  $\lambda = (\lambda_1, \cdots, \lambda_n)$ are positive numbers such that $\prod{\lambda_i} = 1$;

\item $\Sigma_{g,n}$ is topological surface with $n$ marked points that we denote by $\{p_1, \cdots, p_n\}$.

\end{itemize} 

\noindent We define

$$ \mathcal{TD}^*_{g,n} = \big\{ \text{dilation structure on }\Sigma_{g,n} \text{with singularities at the marked points} \big\}/_{isotopies} $$

$$ \mathcal{MD}^*_{g,n} = \big\{ \text{dilation structure on }\Sigma_{g,n} \text{with singularities at the marked points} \big\}/_{diffeomorphisms} $$

\noindent Both $ \mathcal{TD}^*_{g,n}$ and $\mathcal{MD}^*_{g,n}$ can be partitioned according to both the dilation factor and the angle  (which is always an integer multiple of $2\pi$) around its singular points. In order to keep this section readable, we are not going to introduce any notation yet for the moduli spaces induced by this partition.

\vspace{2mm} \noindent Another important remark is that within both $\mathcal{TD}^*_{g,n}$ and $\mathcal{MD}^*_{g,n}$ lie a remarkable locus which is the set of dilation surfaces admitting geodesic triangulations.  For some reasons of a dynamical nature, we believe that this locus is more interesting to study. We therefore define

\begin{equation*}
\begin{aligned}
\mathcal{TD}_{g,n} =  &\big\{ \text{dilation structure on }\Sigma_{g,n} \text{with singularities at the marked points }   \\ 
 & \text{admitting a geodesic triangulation} \big\}/_{isotopies}
\end{aligned}
\end{equation*}  

\begin{equation*}
\begin{aligned}
\mathcal{MD}_{g,n} =  &\big\{ \text{dilation structure on }\Sigma_{g,n} \text{with singularities at the marked points }   \\ 
 & \text{admitting a geodesic triangulation} \big\}/_{homeomorphisms}\end{aligned}
\end{equation*}  

\noindent It is immediate from the definition that $\mathcal{MD}_{g,n}$(resp. $\mathcal{MD}^*_{g,n}$) is the quotient of $\mathcal{TD}_{g,n}$ (resp. $\mathcal{TD}^*_{g,n}$) by the action of the pure mapping class group $\mathrm{MCG}(g,n)$. All these moduli spaces are orbifolds, as ensured by a theorem of Veech.

\begin{theorem}[Veech, \cite{Veech}]
 $\mathcal{TD}_{g,n}$  and $\mathcal{TD}^*_{g,n}$ ($\mathcal{MD}_{g,n}$ and  $\mathcal{MD}^*_{g,n}$) are analytic manifolds (resp. orbifolds) of dimension $ 6(g-1) + 3n$.
\end{theorem}

\subsection{Isoholonomic foliation}

Since $\T$ is a set of marked dilation structures, the following \textit{linear holonomy map} 
$$
\begin{array}{ccccc}
\mathrm{H} & : & \T & \longrightarrow & \mathrm{H}^1(\Sigma_{g,n}, \R^*) \\
 & & \Sigma & \longmapsto & \rho(\Sigma) 
\end{array}
$$ is well defined. It is a submersion according to a theorem of Veech (\cite{Veech}, p.625 Theorem 7.4) hence its level sets define a trivial foliation of $\T$. Since the map $\mathrm{H}$ is equivariant for the mapping class group action on $\T$ with respect to its linear action on  $\mathrm{H}^1(\Sigma_{g,n}, \R^*)$, this foliation passes to the quotient $\T / \mathrm{Mod}(\Sigma_{g,n}) = \M$. This foliation we call the \textit{isoholonomic foliation}.

\subsection{Action of $\mathrm{SL}(2,\mathbb{R})$}
We define in this paragraph an action of $\mathrm{SL}(2,\mathbb{R})$ on $\T$ and $\M$. Consider a dilation atlas $(U, \varphi_U)_{U \in \mathcal{U}}$ on $\Sigma_{g,n}$ where $\mathcal{U}$ is a collection of open subsets of $\Sigma_{g,n}$ and for all $U \in \mathcal{U}$, $\varphi_U : U \rightarrow \C$ is a homeomorphism defining the dilation structure. Let $A$ be an element of $\mathrm{SL}(2,\R)$. One easily verifies that $(U, A \circ \varphi_U)_{U \in \mathcal{U}}$ also define a dilation atlas and hence a new dilation structure on $\Sigma_{g,n}$. By this mean, we define an action of $\mathrm{SL}(2, \R)$  on both $\mathcal{TD}^*_{g,n}$ and $\mathcal{MD}^*_{g,n}$. 

\begin{itemize}
\item The image of geodesic triangulation by the action of an element of $\mathrm{SL}(2,\R)$ is a geodesic triangulation. This action therefore preserves $\T \subset  \mathcal{TD}^*_{g,n}$ and $\M \subset \mathcal{MD}^*_{g,n}$ and defines by restriction an action on both $\T$ and $\MT$. We will only be concerned with this action in the sequel.  

\item This action on $\T$ and $\M$ is locally free (but is not on  $ \mathcal{TD}^*_{g,n}$ and $\mathcal{MD}^*_{g,n}$, see \cite{DFG}). 

\item It preserves the isoholonomic foliation. 

\end{itemize}

\noindent This action restricted to the subgroup of diagonal matrices $\big\{ \begin{pmatrix}
e^t & 0 \\ 0 &e^{-t} 
\end{pmatrix} \ | \ t \in \R \big\}$ defines a flow that we call \textit{Teichmüller flow}.

\subsection{$\Theta$ seen as a function on the moduli space}

The invariant $\Theta$ of a dilation surface defined in Section \ref{theta} actually defines a continuous function:

$$ \Theta : \mathcal{MD}^*_{g,n} \longrightarrow \mathbb{R}_+.$$  A rigorous proof of the continuity would deserve a general discussion on the topology of $\mathcal{MD}^*_{g,n}$ which one will find in \cite{Veech}. The idea though is rather simple. There exists natural coordinates on $\mathcal{MD}^*_{g,n}$ (that one should probably call \textit{'period coordinates'} in reference to the case of translation surfaces) which make the natural (local) identification of $\mathcal{MD}^*_{g,n}$ with spaces of polygons continuous. From such polygonal representations it is easy to see that any embedded cylinder in a dilation surface survives small deformations of such a surface and that the angle varies continuously. We nonetheless spare the reader a detailed discussion on the topology of $\mathcal{MD}^*_{g,n}$ as in the case of dilation tori with two singularities, this topology is completely explicit. 

\section{Tori with two singularities}
\label{twosingularities}

\noindent From now on and until the end of this article we will only work with $\T$ and $\M$ and forget about non-triangulable dilation surfaces.  We will also restrict our attention to the case $g=1$ and $n=2$, namely to tori with two singularities. 

\vspace{2mm}

\noindent For the remainder of the article, $T$ is a torus (thought of as a topological surface), $p_1$ and $p_2$ two marked points on $T$ and $T^*$ is $T \setminus \{ p_1, p_2 \}$. All dilation structures will be thought of as structures on the underlying surface $T$ with singularities at $p_1$ and $p_2$. For any dilation structure on $T$, the angle around any singular point is necessary $2\pi$. Thus the singularity type of a dilation torus is completely determined by the dilation factor around $p_1$ and $p_2$ which we denote by $\lambda_1$ and $\lambda_2 \in \R^*_+$. Since $\lambda_1 \lambda_2 = 1$ we only need to know $\lambda_1 = \lambda$. Without loss of generality we can suppose that $\lambda > 1$. 

\vspace{2mm}

\noindent As of now we denote by $$\TT \subset \mathcal{TD}_{1,2} \  \text{(resp.} \  \MT \subset \mathcal{MD}_{1,2}\text{)}$$ the set of triangulable dilation tori with two singularities whose dilation factor at $p_1$ is $\lambda >1$. Note that both $\TT$ and $\MT$ are saturated sets of the isoholonomic folation and are  $\mathrm{SL}(2,\mathbb{R})$-invariant so we can freely speak of these two objects restricted to $\TT$ and $\MT$ .

\subsection{Invariant measure on the moduli space}

In this particular case, it happens that the isoholonomic foliation and the one induced by the $\mathrm{SL}(2,\R)$ agree, simply because the leaves have same dimension. This fact makes it easy to build an invariant measure for the $\mathrm{SL}(2,\mathbb{R})$-action. 

\vspace{3mm}

\noindent Consider $(a,b)$ a symplectic basis of $T$. We can complete it into a basis of $\mathrm{H}_1(T^*, \Z)$ by adding $c$ the class of a loop turning around $p_1$. The linear holonomy of a dilation structure in $\TT$ is a group homomorphism $\rho :  \mathrm{H}_1(T^*, \Z) \longrightarrow \R^*_+$ such that $\rho(c) = \lambda$. We denote by $\rho_a$ and $\rho_b$ the value of $\rho$ on $a$ and $b$. The symplectic form $$ \omega = \mathrm{d}\log \rho_a \wedge \mathrm{d}\log \rho_b $$ defines a volume on the affine subspace of $\mathrm{H}_1(T^*, \R^*_+)$ of elements $\rho$ such that $\rho(c) = \lambda$. This  form is invariant by the action of the pure mapping class group and hence defines a symplectic\footnote{This form is not symplectic in the usual sense of the term. It is a closed $2$-form vanishing when restricted to the isoholonomic foliation and which is symplectic on two manifold which is transverse to the foliation} form transverse to the isoholonomic foliation of $\MT$. 

\vspace{3mm}

\noindent On each leaf of the isoholonomic foliation we can put a measure which is the trace of a Haar measure of $\mathrm{SL}(2, \R)$. This family of measure coupled to the transverse structure form a measure $\mu$ on the total space $\MT$ which is by definition $\mathrm{SL}(2,\R)$-invariant. This measure is in the same class as the Lebesgue measure.

\subsection{Polygonal models}
\label{subsecpolmod}
We made the hypothesis that we are working exclusively with triangulable dilation tori. From any triangulation we can extract a \textit{pseudo-polygonal} model of a given dilation surface $\Sigma$ (a torus as it happens). In order to achieve this one has to consider a sub-graph $\Gamma$ in the $1$-skeleton of this triangulation which is maximal with respect to the property that $\Sigma \setminus \Gamma$ is connected. For such a $\Gamma$, $\Sigma \setminus \Gamma$ is simply connected and one can consider 

$$ \mathrm{D} : \Sigma \setminus \Gamma \longrightarrow \C $$ the developing map of the dilation structure. It is a \textit{pseudo-polygon}, an immersion of the disk which extends to its boundary and which is piecewise geodesic restricted to this boundary. Moreover, since this pseudo-polygon comes from a dilation structure, sides of the boundary that are identified in $\Sigma$ must be parallel.

\vspace{2mm}

\noindent Because of Euler characteristic considerations, the number of sides of pseudo-polygons only depends on $g$ and $n$. In the particular case of a torus $T$ with two singularities, an associated pseudo-polygon must have six sides which project onto a graph $\Gamma \subset T$ with three edges. 

\noindent To every such pseudo-polygonal model is associated a \textit{gluing pattern} which is the datum of the sides which are glued together. On the Figure \ref{gluingpatterns} below are represented the two possible gluing patterns for a dilation torus. 

\begin{figure}[!h]
  \centering
  \includegraphics[scale=0.4]{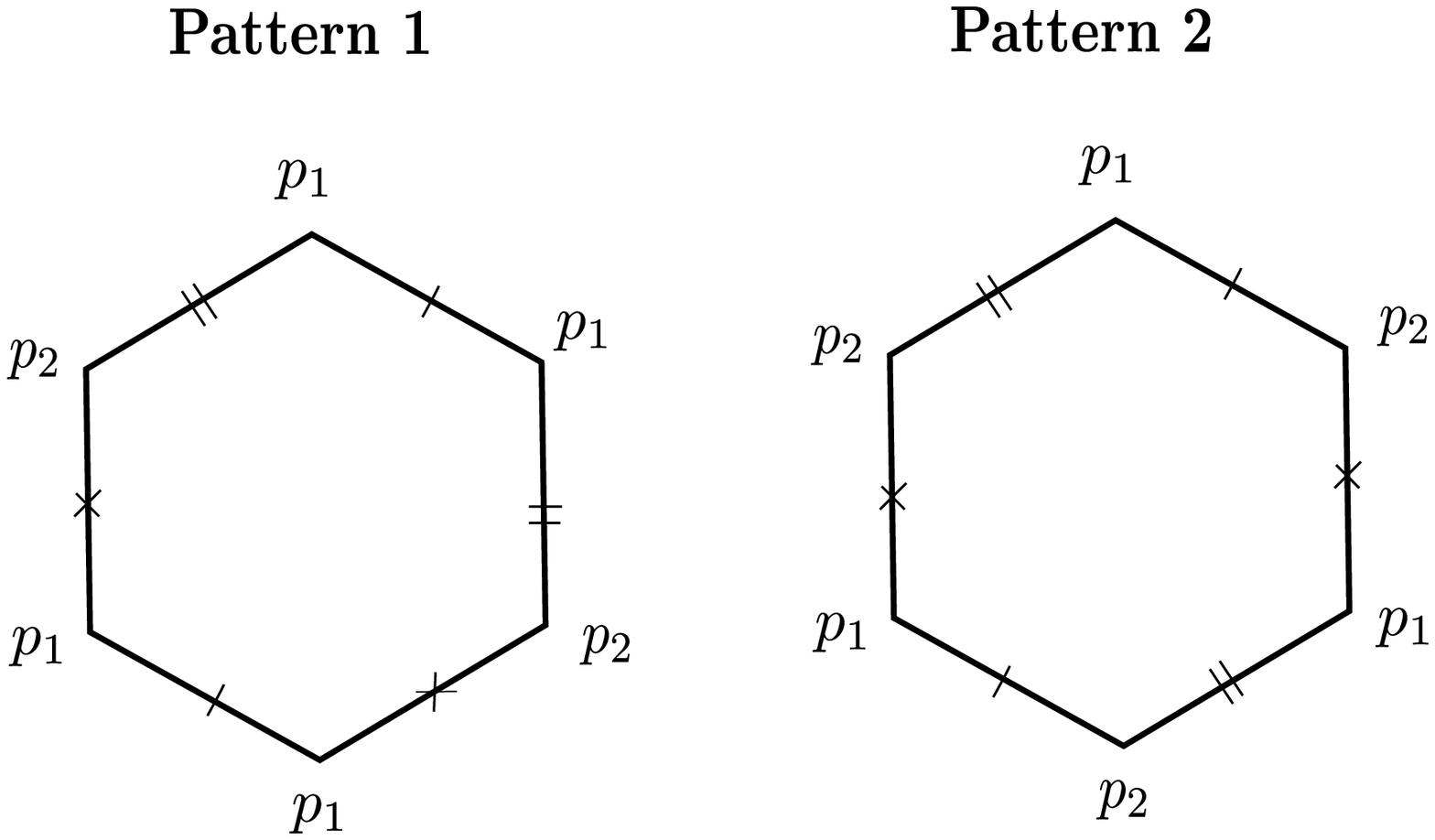}
  \caption{}
  \label{gluingpatterns}
\end{figure}

\noindent By cutting a triangle and pasting it using one of the identifications, one can easily move from Pattern $2$ of Figure 1 to Pattern $1$. We  have the following

\begin{lemma}
\label{polygon}
A pseudo-polygon with six sides such that each side is parallel to another one is actually a regular polygon (meaning that the immersion in definition is actually an embedding).

\end{lemma}

\begin{proof}

We first give the proof for Pattern $1$. We are going to prove that the image the developing map $D$ of the boundary of $\Sigma \setminus \Gamma$ is a non-degenerated polygon which implies the Lemma. Pick the two sides which are paired and opposite to each other. Their image under  $D$ are two parallel segments $A$ and $A'$. The two sides issued from each of these sides are parallel (say $B$ and $B'$ are issued from the extremities of $A$ and $C$ and $C'$ are issued from the extremities of $A'$). The sides $B$ and $C$ must meet and if they do so must $B'$ and $C'$ since $B$ and $B'$ are parallel and $C$ and $C'$ are parallel. But in that case $A \cup A' \cup B \cup B' \cup C \cup C'$ is the boundary of a non-degenerated polygon.

\noindent A cut and paste operation transforms and pseudo-polygonal model with Pattern $2$ to one with Pattern $1$. One can apply what we have just proved in the preceding paragraph to get that this new pseudo-polygonal model is non-degenerated. We get back to Pattern $2$ doing the opposite cut and paste operation. Just notice that cutting and pasting any triangle from a hexagonal model whose paired sides are parallel gives rise to a non-degenerated polygonal model. 

\end{proof}

\noindent Lemma \ref{polygon} therefore ensures that any torus with two singularities is built out of a proper hexagon. This fact is going to be used extensively in the sequel.

\subsection{Decompositions in cylinders}
\label{subsecdecomposition}

Informally, a \textit{decomposition in cylinders} is a way to build a dilation torus by gluing two dilation cylinders along their boundaries. As it would be a bit painful to give a precise description of the gluing operation that one needs to perform to obtain a well-defined dilation surface, we will use another definition. 

\begin{defi}
A decomposition in cylinders of dilation torus with two singularities is the datum of two disjoint closed saddle connections.
\end{defi}

\noindent It would be a good exercise of affine geometry to show that the complement in a dilation torus of two such saddle connections is actually the union of two dilation cylinders of same angle.

\begin{proposition}

\begin{enumerate}
\item Every element of $\MT$ can be decomposed into cylinders.
\item The dynamics in every direction belonging to one of the two cylinders is Morse-Smale. 

\item Conversely, every Morse-Smale direction corresponds to such a cylinder decomposition.

\end{enumerate}
\end{proposition}

\begin{proof}
Let $T$ be an element of $\MT$. The points $(2)$ and $(3)$ are a direct consequence of the first one hence we only prove $(1)$. Consider $H$ a hexagonal model for $T$ with Pattern $1$. There is a unique pair of identified sides which project onto a closed saddle connection. Consider the two other pairs of identified sides. Each of these bound a quadrilateral in $H$ which project onto a cylinder in $T$. We get by this mean two cylinders which form a cylinder decomposition of $T$.

\end{proof}

\noindent Another consequence of the form of the decomposition in cylinders is the following:

\begin{lemma}
For every $T \in \MT$, the multiplier of any embedded cylinder belongs to $ ]1, \lambda[ $.
\end{lemma}

\begin{proof}

Any cylinder participate to a decomposition in cylinders. One can take a curve along which one can compute the linear holonomy of the first cylinder  and slide it over one of the singular points to get a curve along which one compute the inverse of the holonomy of the second cylinder. The multipliers $\rho_a$ and $\rho_b$ of the two respective cylinders $C_a$ and $C_b$ thus satisfy the following 

$$ \rho_a \rho_b = \lambda.$$ The fact $\rho_a$ and $\rho_b$ are by definition greater or equal to $1$ implies the Lemma.
\end{proof}

\paragraph*{\bf The canonical polygonal model for a decomposition}

We define in this subsection a polygonal model associated to a decomposition in cylinders of dilation torus $T$ which is in a sense of "minimal complexity".  

\vspace{3mm}

\noindent Let $T$ be a dilation torus. We consider a cylinder decomposition of $T$.
\begin{itemize}
\item  $C_a$ and $C_b$ are the two cylinders;
\item  $l_1$ (respectively $l_2$) the closed saddle connection based at $p_1$ (respectively $p_2$);
\item $\rho_a$ (respectively $\rho_b$) is the dilation factor of $C_a$ (respectively $C_b$).
\end{itemize}
\noindent Without loss of generality, we can suppose that $\rho_a < \rho_b$. Choosing a polygonal model with Pattern $1$ associated with this decomposition amounts to choosing two saddle connections $l_a$ and $l_b$ joining $p_1$ and $p_2$ whose interiors are respectively contained in $C_a$ and $C_b$. We give a canonical way to choose such a pair of saddle connections. 

\noindent Let $v$ be the unique segment issued from $p_1$ in the direction perpendicular to the one of $l_1$, whose end point is on $l_2$ and whose interior is contained in $C_a$. There is a unique saddle connection $s_1$ in $C_a$ from $p_1$ to $p_2$ which does not intersect $v$ and such that the angle between $v$ and $s_1$ is non-negative. 
\noindent Similarly we consider $v'$ the unique segment issued from $p_2$ in the direction perpendicular to the one of $l_2$, whose end point is on $l_1$ and whose interior is contained in $C_b$. There is a unique saddle connection $s_2$ in $C_b$ from $p_2$ to $p_1$ which does not intersect $v'$ and such that the angle between $v'$ and $s_2$ is non-negative.

\begin{figure}[!h]
\centering
\psfrag{p1}[][][0.8]{$p_1$}
\psfrag{p2}[][][0.8]{$p_2$}
\psfrag{s1}[][][0.8]{$s_1$}
\psfrag{s2}[][][0.8]{$s_2$}
\psfrag{v}[][][0.8]{$v$}
\psfrag{v}[][][0.8]{$v'$}
\includegraphics[scale=0.31]{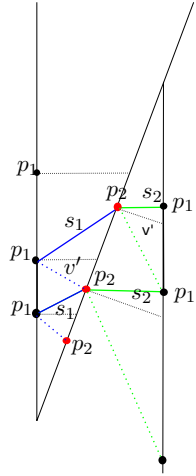}
\caption{A canonical model associated to a decomposition}
\label{canonical}
\end{figure}


\begin{defi}
We call canonical polygonal model associated to the decomposition in cylinders associated to $C_a$ and $C_b$ the one given by cutting along $l_1$, $s_1$ and $s_2$. 
\end{defi}

\noindent Such a decomposition has the pleasant property that the dilation gluing factor between the two sides of the associated hexagon that project onto $l_1$ is bounded above by $\rho_a < \lambda$ and below by $\rho_b^{-1} > \lambda^{-1}$.

\subsection{Directional foliations and piecewise affine homeomorphisms of the circle}

We give a classification of the possible dynamical behaviours for the directional foliations of dilation tori with two singularities. 

\begin{proposition}
\label{classification}
The vertical foliation of a triangulable dilation tori with two singularities is of one the following type:

\begin{enumerate}

\item It has one attracting leaf and one repelling leaf which are respectively the $\omega$-limit and the $\alpha$-limit of every other leave. 

\item It has one attracting closed saddle connection and one repelling closed saddle connection which are respectively the $\omega$-limit and the $\alpha$-limit of every other leaf. 

\item It is completely periodic.

\item It is minimal.

\end{enumerate}

\end{proposition}

\begin{proof}
If $T$ is endowed with a triangulable dilation structure, then its vertical foliation does not have any Reeb component for otherwise it would contain a cylinder of angle at least $\pi$. Hence there exists a simple closed curve in $T$ which is transverse to the vertical foliation and the first return map to this curve defines a piecewise affine homeomorphism of the circle $f: S^1 \longrightarrow S^1$.  We distinguish on whether the rotation number of $f$ is irrational or not. 

\noindent If it is irrational, $f$ is semi-conjugated to the a minimal rotation. Since $f$ is piecewise affine, its derivative has bounded variations and Denjoy theorem applies: the semi-conjugation is an actual conjugation. Hence $f$ is minimal and so is the vertical foliation.

\noindent Otherwise, the rotation number of $f$ is rational and $f$ has a periodic orbit. Assume the closed leaf associated with this periodic orbit does not meet any singular point. This leaf is either hyperbolic or parabolic (meaning that the derivative of the first return map of the flow near this periodic is either different or equal to $1$). In the first case, this leaf is contained in a cylinder itself being part of a cylinder decomposition in which case we are in  case (1). 

\noindent Finally if the foliation has a closed saddle connection then it bounds a cylinder which is either flat and in which case the foliation is totally periodic(this case is the same as above) or it is a dilation cylinder in which case we are in case (2). 

\end{proof}

\section{The geometry of dilation tori}
\label{geometry}

\subsection{Dilation modulus of a dilation cylinder.}

We introduce another geometric quantity associated to a cylinder which we call its \textit{dilation modulus}.

\begin{defi}
Let $\mathcal{C}$ be a dilation cylinder of angle $\theta < \frac{\pi}{2}$ and of dilation factor $\rho > 1$. Its dilation modulus (or simply modulus) is the quantity 

$$ \frac{\tan \theta}{\rho - 1}.$$  
\end{defi}

\noindent If we are dealing with a flat cylinder, we will use the standard notion of modulus that is the ratio between its length and width.

\noindent The dilation modulus measures how 'long and thin' a cylinder is. Indeed assume you form a cylinder by gluing two opposite sides $A$ and $A'$ of a quadrilateral and that one of the two sides which are not glued together, say $B$, is perpendicular to $A$ and $A'$. Assume that $A$ is shorter that $A'$. Then the modulus of this cylinder is equal to the length of $B$ divided by the length of $A$.
\noindent Notice that this notion of modulus is only defined for cylinders whose angle is less than $\frac{\pi}{2}$. Hence this 'long and thin' heuristic only applies when the angle is less than $\frac{\pi}{2}$.

\vspace{2mm}

\noindent The reason why we care about this quantity is because cylinders of large modulus force the function $\Theta$ to be small. Indeed we have the following lemma:

\begin{lemma}
\label{smallangle}
Let $T$ be a dilation torus with two singular points. There exists a constant $\kappa$ such that the following holds: assume $T$ has a cylinder $\mathcal{C}$ of modulus $M$ and of angle $\theta$ then 

$$\Theta(T) \leq \max \big\{ \theta, \frac{\kappa}{M} \big\} $$
\end{lemma}

\begin{proof}
We give a quick sketch of a proof as it is very elementary. Consider any other cylinder $\mathcal{D}$ in $T$ of angle $\alpha$. Either it is the other cylinder of the decomposition to which $\mathcal{C}$ participate in or $\mathcal{D}$ must intersect $\mathcal{C}$. Both boundary components of $\mathcal{D}$ must intersect one of the boundary component of $\mathcal{C}$ (they 'enter' $\mathcal{C}$) and one sees an 'angular strip' (the two lines on the boundary of this strip form an angle $\alpha$ entering $\mathcal{C}$ by one of its boundary component and exiting through the other). This strip does not contain any singular point in its interior. If the modulus is large this forces $\alpha$ to be very small (less than $\arctan \frac{1}{M}$) and a compacity argument plus the fact that $\arctan(\epsilon) \sim \epsilon$ when $\epsilon << 1$ imply the Lemma.
\end{proof}

\subsection{Convex polygonal models}

The rest of this Section is dedicated to prove a weak converse to Lemma \ref{smallangle}: tori $T$ for which $\Theta(T)$ is small must contain a cylinder of not too small modulus. The precise statement is the one of Lemma \ref{largemodulus}. To get to this point we prove the existence of well-behaved polygonal models.

\begin{lemma}
\label{convexexist}
Every dilation torus in $\MT$ can be represented by a convex hexagon with Pattern $2$.
\end{lemma}

\begin{proof}
Start with any polygonal model with Pattern $1$. First we point out that a hexagonal model with Pattern $1$ for a dilation torus can hardly ever be convex. Indeed the two angles at vertices projecting onto $p_2$ (see Figure \ref{convex}) add up to $2\pi$ and therefore the only way it can be convex is if those angles are both equal to $\pi$. 

\vspace{2mm} \noindent We describe an algorithm where each step consists in cutting a triangle from the hexagon and gluing it elsewhere and which produces in finitely many step the required convex model. 

\begin{figure}[!h]

\psfrag{p1}[][][0.8]{$p_1$}
\psfrag{p2}[][][0.8]{$p_2$}
\psfrag{T}[][][0.8]{$T$}
  \centering
  \includegraphics[scale=0.4]{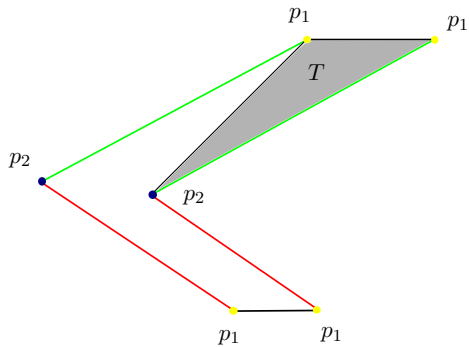}
  \caption{A hexagonal model with Pattern $1$.}
  \label{convex}
\end{figure}

\noindent The Figure \ref{convex} above displays a hexagon $H$  for which vertices of same color (blue or yellow) project onto the same singular point and sides of same color are parallel and identified. Form a triangle $T$ whose vertices are the blue vertex at which the interior angle is greater than $\pi$ and two consecutive yellow vertices (there are essentially two choices for such a triangle). Without loss of generality suppose that one of the green sides of $H$ is a side of $T$. We cut $T$ and are left with a pentagon. 
\vspace{2mm} \noindent At this point it can be that the angle at the blue vertex is still larger than $\pi$. In that case we glue $T$ to the other green side and we get a new polygonal model for our torus with Pattern $1$.  We repeat this operation until this angle is less than $\pi$. When we have achieved this, we simply glue the triangle $T$ to the other black side. We get a new polygonal model with Pattern $2$ and we claim that this one is convex. Indeed, the construction ensures that three consecutive angles are less than $\pi$. Because identified sides are parallel, opposites angles for a hexagonal model with gluing pattern $2$ are equal hence every interior angle is less than $\pi$ thus this hexagon is convex.

\end{proof}

\noindent We now exploit the existence of such convex hexagonal models to prove that we can find a hexagonal model for which at least two pairs of identified sides are not too large. In the sequel we consider convex hexagonal models with Pattern $2$ (see Section \ref{subsecpolmod}) \textbf{of area} $1$. We will also denote by $l_1, l_2$ and $l_3$ the respective lengths of the shortest of each pair of sides identified.

\begin{lemma}
\label{goodconvex}
There exists a constant $K_{\lambda}$ such that every dilation torus in $\MT$ can be represented by a convex hexagon of area $1$ with Pattern $2$ and such that $l_1$ and $l_2$ are less than $K_{\lambda}$.
\end{lemma}

\begin{proof}
Start with a convex polygonal model $H$ given by Lemma \ref{convexexist}. We first need to make two preliminary remarks.

\begin{enumerate}

\item By virtue of $H$ being convex, the quadrilateral formed by joining the ends of two sides identified by the gluing is contained within $H$ hence projects onto a cylinder in the associated torus. Such a cylinder has dilation factor at most $\lambda$ hence the ratio between side glued together is at most $\lambda$. 

\item Consider any pair of  opposite vertices $q$ anr $r$. The segment $l$ joining these two cuts $H$ into two quadrilaterals $Q_1$ and $Q_2$. Call $a$ and $b$ the two sides of $Q_1$ which are adjacent to $l$. Both of them are identified to $a'$ and $b'$ in $Q_2$. We claim that gluing $Q_1$ and $Q_2$ along either $a$ and $a'$ or $b$ and $b'$ yields a convex hexagonal model with gluing pattern $2$. The fact that the gluing pattern is number $2$ is straightforward. Let us denote by $\alpha$ the angle of $H$ at $r$ and $\beta$ the angle at $q$. The segment $l$ cuts $\alpha$ into two angles $\alpha_1$ and $\alpha_2$ and $\beta$ into two angles $\beta_1$ and $\beta_2$. Up to relabelling we can assume one of the new angle induced by the two suggested gluing are respectively $\alpha_1 + \beta_1$ and $\alpha_2 + \beta_2$. Since $\alpha + \beta \leq 2\pi$ at least one of these two is less than $\pi$. We choose the gluing realising this and for this one the new hexagon is convex since it has three consecutive angles less than $\pi$ hence all of them by virtue of the symmetry in the angles discussed in the proof of Lemma \ref{convexexist}.

\end{enumerate}

\noindent We give a proof of qualitative nature in order not to cloud the idea with too many quantifiers. We invite the interested reader to fill in the gaps as an exercise. Assume $l_1$, $l_2$ and $l_3$ are very large. Pick any segment $l$ joining opposite vertices in $H$. It bounds a quadrilateral whose area is less than $1$ and whose three other sides are very long. It therefore must be small. We perform the gluing explained in the second remark above two get a new convex hexagon with only two segment which are very large (since the ratio between segments paired together is bounded, the rescaling to make this new hexagon of area one is bounded as well and this does not affect the fact that $H$ has a pair of identified sides of moderate length). 

\noindent Reiterate this procedure using the line joining the vertices at which the remaining two long pairs of sides meet to get the desired polygonal model.

\end{proof}

\noindent We end this section by proving the following Lemma which predicts the existence of cylinders of not too small affine modulus.

\begin{lemma}
\label{largemodulus}
There exists a constant $M_{\lambda} > 0$ such that every dilation torus $T$  in $\MT$ such that $\Theta(T) < \frac{\pi}{2}$ contains a cylinder of modulus greater than $M_{\lambda}$.
\end{lemma}

\begin{proof}

Pick a polygonal model for $T$ given by Lemma \ref{goodconvex}. If $l_3$ the possibly very large side is indeed very large then the cylinder one gets when projecting the quadrilateral defined by the long sides identified has large modulus. Indeed this cylinder is going to be very long and thin because $H$ has fixed area $1$. Since we have assumed that $\Theta(T) < \frac{\pi}{2}$ being long thin implies that the modulus is large. Precisely, there exists $K>0$ such that if $l_3$ is larger than $K$ then it contains a cylinder of modulus larger than a certain constant $M_1$. Otherwise the length $l_1$, $l_2$ and $l_3$ are bounded by a simple compactness argument we get the existence of a cylinder of modulus larger than $M_2$. Then $M_{\lambda} = \min\{M_1, M_2 \}$ works.

\end{proof}

\section{Volume of the cusp}
\label{volume}
\noindent We discuss in this section the geometry of $\MT$. In particular we show that regions of $\MT$ where the angle function $\Theta$ is bounded from above have finite volume. This section is dedicated to the proof of the following result:

\begin{theorem}
\label{finitevolume}
The volume of the 'cusp'
$$\mathcal{C} =  \big\{ T \in \MT \ \big| \ \Theta(T) \leq \frac{\pi}{4}  \big\}$$ is finite.
\end{theorem}

\subsection{Haar measure of $\SL$}

We first recall an explicit description of the Haar measure of $\SL$. We denote by $N$ the subgroup of $\SL$ of unipotent lower triangular matrices, $A$ the subgroup of diagonal matrices and $K$ the subgroup of orthogonal matrices.

\begin{lemma}

Every element  $ m \in \SL$ can be decomposed in a unique way as a product of the form

$$ m =  k \cdot n \cdot a $$ where $n = \begin{pmatrix}
1 & 0 \\
t & 1 
\end{pmatrix} \in N$, $a =  \begin{pmatrix}
\nu & 0 \\
0 & \nu^{-1} 
\end{pmatrix} \in A$ and $k =  \begin{pmatrix}
\cos \theta & -\sin \theta \\
\sin \theta & \cos \theta 
\end{pmatrix} \in K$.

\end{lemma} \noindent  We leave the proof of this classical Lemma to the reader. This gives us local coordinates $(t, \nu, \theta)$ on $\SL$. In these coordinates, the Haar measure of $\SL$ is up to multiplication to a positive constant

$$dm_{\SL} =  \frac{d\nu}{\nu}dt d\theta.$$ We will make extensive use of this fact in the sequel.

\subsection{Computation of the volume}

In this section we consider hexagonal models with gluing pattern $1$. Fix $1 < \rho_a < \lambda$, $\lambda^{-1} < \rho_b <1$ and $ \rho_b < \rho_c < \rho_a$. We use the notation $\rho = (\rho_a, \rho_b, \rho_c)$. There is a unique hexagonal model $H_{\rho,M}$ with gluing pattern $1$ such that

\begin{itemize}
\item the associated cylinders $C_a$ and $C_b$ have multipliers $\rho_a$ and $\rho_b$ respectively;
\item the modulus of $C_a$ is exactly $M$;
\item the unique pair of sides which project onto a closed saddle connection are glued along a dilation of factor $\rho_C$.
\end{itemize}

\noindent The idea that is going to drive the computation to come is that every canonical hexagonal model can be obtained from a $H_{\rho,M}$ after applying a matrix in $\mathrm{SL}(2,\R)$ of the form $\begin{pmatrix}
\mu & 0 \\
t & \mu^{-1} 
\end{pmatrix}$ with $t$. 

\vspace{3mm}

\noindent Consider the $1$-parameter (semi-simple) subgroup of $\mathrm{SL}(2,R)$ which preserves the two directions of the boundary components of the cylinder decomposition corresponding to $H_{\rho}$. Without loss of generality, we can suppose that the direction of the pair of sides glued along the dilation of factor  $\rho_C$ is the vertical one. In this case this one parameter subgroup is 

$$  \big\{ q_t = \begin{pmatrix}
\sqrt{\rho_a}e^t & 0 \\
 \frac{\rho_ae^{2t} - 1}{\sqrt{\rho_a}e^t  \tan \theta_a} & \sqrt[-2]{\rho_a}e^{-t}

\end{pmatrix} \ | \ t \in \mathbb{R} \big\} $$ where $\theta_a$ is the angle of $C_a$.  The image of $C_a$ under the action of $q_1$ is represented on Figure \ref{twist} below: plain lines in green and blue are mapped by $q_1$ to the dashed ones of same colour. Formally $q_1$ realises a Dehn twist in the only essential simple closed curve of $C_a$.  

\begin{figure}[!h]

\psfrag{1}[][][0.8]{$1$}
\psfrag{r}[][][0.8]{$\rho$}

\centering
\includegraphics[scale=0.31]{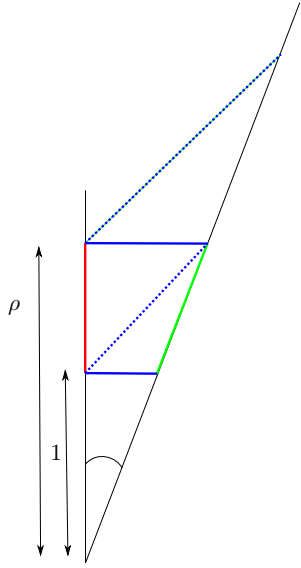}
\caption{The image of a polygonal model for $C_a$ by the action of $q_1$.}
\label{twist}
\end{figure}


\noindent Building on these remarks we notice that $\{ q_t \cdot H_{\rho} \ | \ t \in [0,1] \}$ contains all canonical hexagonal models with gluing pattern $1$, gluing factors $\rho$ and whose cylinder $C_a$ has modulus $M$. 

\noindent Notice also that $\begin{pmatrix}
e^{\frac{t}{2}} & 0 \\
0 & e^{-\frac{t}{2}}
\end{pmatrix} \cdot H_{\rho,M} = H_{\rho, e^tM}$.  As a consequence of these two facts

$$  \big\{ \begin{pmatrix}
\sqrt{\rho_a}e^t & 0 \\
 \frac{\rho_ae^{2t} - 1}{\sqrt{\rho_a}e^t \nu(\rho_a-1)M_{\lambda}} & \sqrt[-2]{\rho_a}e^{-t}

\end{pmatrix}  \begin{pmatrix} \sqrt{\nu} & 0 \\ 
0 & \sqrt{\nu}^{-1} 
\end{pmatrix} \cdot H_{\rho, M_{\lambda}} \ | \ t\in [0,1] \ \text{and} \  \nu \in [1, + \infty[ \big\} $$ contains every possible canonical hexagonal model with gluing factors $\rho$ and whose cylinder $C_a$ has modulus less than $M_{\lambda}$.

\noindent In particular there exists a constant $K>0$ depending only on $\lambda$ such that 

\begin{equation}
\label{domain}
\big\{ \begin{pmatrix}
1 & 0\\
t & 1

\end{pmatrix}  \begin{pmatrix} \sqrt{\nu} & 0 \\ 
0 & \sqrt{\nu}^{-1} 
\end{pmatrix} \cdot H_{\rho, M_{\lambda}} \ | \ \nu \in [1, + \infty[ \  \ \text{and} \  t\leq \frac{K}{\nu}  \big\} \end{equation} also contains every possible canonical hexagonal model with gluing factors $\rho$ and whose cylinder $C_a$ has modulus less than $M_{\lambda}$. 

\vspace{2mm}

\noindent We recall that every dilation tori whose angle is less than $\frac{\pi}{4}$ can be represented by an canonical hexagonal model. For such a canonical model we have seen (see Section \ref{subsecdecomposition}) that $1 \leq \rho_a \leq \lambda$ and $\lambda^{-1}\leq \rho_c \leq \lambda$. Recall that $\mathcal{C }=  \big\{ T \in \MT \ \big| \ \Theta(T) \leq \frac{\pi}{4}  \big\}$. Let us denote by $\mathcal{H}_{\rho}$ the set of hexagonal models whose cylinder $C_a$ has modulus larger than $M_{\lambda}$. Every set $\mathcal{H}_{\rho}$ naturally projects to $\MT$. Every such projection is a local diffeomorphism onto a leaf of the isoholonomic foliation. Moreover, each $\mathcal{H}_{\rho}$ naturally identifies with an open subset of $\mathrm{SL}(2, \mathbb{R})$ to give local coordinates 

\begin{equation}
\label{coordinates}
 \begin{array}{ccc}
[\lambda^{-1}, \lambda]\times [\lambda^{-1}, \lambda] \times \SL  & \longrightarrow & \MT \\
(\rho_a, \rho_c, A) &  \longmapsto  & A\cdot H_{\rho}
\end{array} \end{equation}  where $\rho = (\rho_a, \rho_a\lambda^{-1}, \rho_c)$ and where we identified $A\cdot H_{\rho}$ with the torus obtained after gluing the corresponding hexagonal model. The volume form $d\mu$ writes in these coordinates 

$$ d\mu = d\log\rho_a \wedge d\log\rho_c \wedge dm_{\SL} $$ where $dm_{\SL}$ is the Haar measure on $\SL$. We have that 

$$ \mathrm{vol}(C) \leq \int_{\text{modulus} \leq M_{\lambda}}{d\mu} $$ By \ref{domain} and \ref{coordinates} we get that

$$ \mathrm{vol}(C) \leq \int_{\lambda^{-1}}^{\lambda}{\int_{\lambda^{-1}}^{\lambda}{\int_{\nu \geq1 , \ t\leq \frac{K}{\nu}, }{dm_{\SL}(\begin{pmatrix}
\cos \theta & -\sin \theta\\
\sin \theta & \cos \theta
\end{pmatrix} \begin{pmatrix}
1 & 0\\
t & 1
\end{pmatrix}  \begin{pmatrix} \sqrt{\nu} & 0 \\ 
0 & \sqrt{\nu}^{-1} 
\end{pmatrix})}d\log\rho_a}d\log\rho_c} $$ 

\noindent The Haar measure of $\SL$ in these coordinates is $\frac{d\nu}{\nu}dtd\theta$ and this integral becomes

$$ \mathrm{vol}(C) \leq \int_{\lambda^{-1}}^{\lambda}{\int_{\lambda^{-1}}^{\lambda}{\big(\int_{\nu \geq1 , \ t\leq \frac{K}{\nu^2}, \ \theta \in S^1}{\frac{d\nu}{\nu}dtd\theta}\big)d\log\rho_a}d\log\rho_c} $$ thus making the change of variable $s = \frac{\nu}{K}t$ we get

$$ \mathrm{vol}(C) \leq 4\pi\log(\lambda)K \int_1^{+\infty}{\frac{d\nu}{\nu^3}}. $$ This proves that the volume of $\mathcal{C}$ is finite and completes the proof of Theorem \ref{finitevolume}. 

\section{Dynamics of the Teichmüller flow}
\label{dynamics}

\noindent In this section we explain how to exploit the finiteness of the volume of the cusp proved in the previous section to show that the set $E \subset \MT$ of dilation tori whose vertical foliation is minimal has measure zero. The strategy we follow is the following: 

\begin{itemize}

\item we first prove that points in $E$ for which the function $\Theta$ does not tend to zero along the orbit of the Teichmüller flow are not density point in the intersection of $E$ with the set of directions of the associated dilation torus;

\item we then show that the set of points in $E$ for which  $\Theta$ tends to zero along the orbit of the Teichmüller flow has measure zero - this is a consequence of the finiteness of the volume of the cusp $\mathcal{C}$. 
\end{itemize} These two points together imply (not completely directly)

\begin{theorem}
\label{maintheorem}
The set of Morse-Smale points in $\MT$ has full measure. Equivalently, $E$ has measure zero.
\end{theorem}

\noindent The equivalence between these two statement is ensured by Proposition \ref{classification}. 

\subsection{Density points in the set of exceptional directions}

Recall that if $A$ is a Lebesgue-measurable subset of
 $\R^n$, $x\in A$ is said to be a density point of $A$ if 
 
 $$ \lim_{r\rightarrow 0}{\frac{\mathrm{Leb}\big(A\cap B(x,r)\big)}{\mathrm{Leb}\big(B(x,r) \big)}} = 1.$$ This definition also makes sense for any measure on a  smooth manifold which is absolutely continuous with respect to the Lebesgue measure. We recall the following lemma

 \begin{lemma}[Lebesgue regularity Lemma]
Let $A$ be a Lebesgue-measurable subset of a smooth manifold $M$. If $A$ does not have measure zero, then almost every point in $A$ is a density point of $A$.
\end{lemma}

\noindent We now move on to give a criterion for a point in $\MT$ to be a density point of the set of exceptional direction. For a dilation torus $T$ we denote by $E_T \subset S^1$ the set of minimal directions on $T$. 

\begin{lemma}
Let $T$ be a dilation torus whose vertical foliation is minimal (meaning that $\frac{\pi}{2} \in E_T$). Suppose that $\Theta(g_t \cdot T)$ does not converge to $0$ as $t$ tends to $+\infty$. Then $\frac{\pi}{2}$ is not a density point in $E_T$. 
\end{lemma}

\begin{proof}
First notice that $E_T$ is invariant by the rotation of angle $\pi$ (a foliation which is minimal induces two exceptional directions: the ones corresponding to the two possible orientations). Hence a point $p\in E_T$ is a density point if and only if $-p = p + \pi$ is a density point. More generally, in this particular case of tori with two singular points directions have same dynamical behaviour as their opposite and we might as well think of $E_T$ as a subset of $\mathbb{RP}^1$.

\vspace{3mm} \noindent The action of $g_t$ on $\mathbb{RP}^1$ is of North-South type: the vertical direction is a repulsive fixed point and the horizontal one is an attracting one. The idea behind this proof is that the Teichmüller flow $g_t$ allows us as we are making $t$ larger and larger to zoom in close to the vertical direction. Directions in $\mathbb{RP}^1$ that participate to a cylinder decomposition lie in the complement of $E_T$.

\vspace{3mm} \noindent By hypothesis there exists a sequence $t_n \rightarrow +\infty$ such that $g_{t_n} \cdot T$ has a cylinder of angle greater than a fixed $\epsilon> 0$. Without loss of generality (up to making $\epsilon$ a bit smaller) we can assume that the directions having a closed leave in this cylinder are direction between $\theta_1(n)$ and $\theta_2(n)$ with $0 <\theta_1(n) < \theta_2(n) \leq \frac{\pi}{2} - \epsilon$. Notice that the projective action of $g_t$ on $[0, \frac{\pi}{2}]$ is a convex map which fixes $0$ and $\frac{\pi}{2}$. It implies that the ratio $\frac{\theta_2(n) - \theta_1(n)}{\frac{\pi}{2} - \theta_1(n)}$ is less that $\frac{g_{-t_n}(\theta_2(n)) - g_{-t_n}(\theta_1(n))}{\frac{\pi}{2} -  g_{-t_n}(\theta_1(n))}$.  But the former is less than $\frac{2\epsilon}{\pi}$ and the latter is less than the proportion of Morse-Smale direction in a ball of radius less than $g_{-t_n}(\frac{\pi}{2}- \epsilon)$. Picking $t_n$ large enough this latter quantity can be made arbitrarily small. Since the proportion of Morse-Smale directions in balls centred at $\frac{\pi}{2}$ does not tend to zero, $\frac{\pi}{2}$ is not a density point in $E_T$. This proves the Lemma.
\end{proof}

\begin{lemma}
\label{measurezero}
The set $D \subset E$ of dilation tori $T$ whose vertical direction is a density point in $E_T$ has measure zero. 
\end{lemma}

\begin{proof}

We simply prove that the set of tori such that $\Theta(g_t \cdot T)$ tends to zero as $t$ tends to infinity has measure zero. This set containing $D$, it will imply the Lemma.
\noindent We assume by contradiction that such a set has positive measure $A>0$. By Theorem \ref{finitevolume}, the set of tori $T$ such that $\Theta(T) \leq \frac{\pi}{2}$ has finite volume. In turn we have that there exists $\theta_A$ such that 
$$ F = \big\{ T \in \MT \ \big| \ \Theta(T) \leq \theta_A  \big\}$$ has measure less that $\frac{A}{3}$. Also there exists $t_0$ such that the set   

$$ D' = \big\{ T \in D \ \big| \ \forall t \geq t_0 \ \Theta(T) < \theta_A  \big\}$$ has measure more than $\frac{2A}{3}$. But by definition $g_t(D') \subset F$ which implies $\mu(g_t(D')) \leq \mu(F)$. The action of $g_t$ being measure-preserving we get that $\mu(D') \leq \mu(F)$ which is a contradiction.
\end{proof}

\subsection{Proof of the main theorem}

\noindent We prove here Theorem \ref{maintheorem}.

\begin{proof}[Proof of Theorem \ref{maintheorem}]

Assume by contradiction that $\mu(E) > 0$ is positive. We introduce $\mathcal{M}(\lambda)$ the moduli space of dilation with no marked direction (we identify in $\MT$ those elements which are image of one another by a rotation in $\mathrm{SO}(2)$). There is a natural projection 

$$ \pi : \MT \longrightarrow  \mathcal{M}(\lambda)$$ which is a fibration whose fiber are circle. For each $T$ in $\MT$, $\pi^{-1}(\pi(T))\simeq S^1$ naturally identifies to the set of directions on $T$. In turn it makes sense for any $T \in \mathcal{M}(\lambda)$ to talk of $E_T$ its set of exceptional directions. This projection being smooth and the measure $\mu$ in the same class as the Lebesgue measure, Fubini theorem ensures that the set of $T \in \mathcal{M}(\lambda)$ such that $E_T$ has positive measure has positive measure itself. By Lebesgue regularity lemma, for each of these $T$ the set of density points in $E_T$ has positive measure which in turn implies that the set of elements $T \in \MT$ whose vertical foliation is a density point in $E_T$ has positive measure. This contradicts Lemma \ref{measurezero} and thus terminates the proof of Theorem \ref{maintheorem}.

\end{proof}

\section{Further comments/open questions}
\label{further}
\noindent We make a few comments and single out questions that seem to be worth further investigation. 
\vspace{3mm}
\paragraph*{\bf The analogy with hyperbolic manifolds of infinite volume}

We believe that there is a interesting analogy to be drawn between moduli spaces of dilation surfaces and hyperbolic manifolds of infinite volume. The table below gives pairing between features of both worlds that support this analogy.  

\begin{center}
\begin{tabular}{| l | l |}
\hline
\textbf{Moduli space of dilation tori} & \textbf{Hyperbolic manifold of infinite volume} \\
\hline
Teichmüller flow & Geodesic flow \\ 
\hline
Mapping class group $\mathrm{MCG}(T*)$ & $\Gamma$ an infinite covolume Kleinian group \\
\hline 
$\{ T \ | \ \theta(T) \geq \frac{\pi}{2}\}$ & Funnel (infinite volume)\\
\hline
$\{ T \ | \ \theta(T) \leq \frac{\pi}{4}\}$ & Cusp (finite volume)\\
\hline
Measure $\mu$ & Liouville measure \\
\hline
$\{ \text{tori with minimal vertical foliation}\}$ & Limit set of $\Gamma$\\
\hline
\end{tabular}
\end{center}

\vspace{3mm}

\paragraph{\bf General case}
We do hope that Theorem \ref{maintheorem} generalises to higher dimensional moduli spaces. We formulate the following two-legged conjecture

\begin{conjecture}
\label{conj}
\begin{itemize}
\item Fix $g$ and $n$ such that $\mathcal{MD}_{g,n}$ is non-empty. The vertical foliation of almost every dilation surface in $\mathcal{MD}_{g,n}$ is Morse-Smale.

\item Fix $k\geq 3$. Almost every $k$-affine interval exchange transformation is Morse-Smale. In particular, almost every piecewise affine circle homeomorphism is Morse-Smale.
\end{itemize}
\end{conjecture}

\noindent A feature of the general case that is to complicate matters is that contrary to the torus case, there are foliations which are neither minimal nor Morse-Smale but of 'exceptional' type, meaning that they have closed invariant sets which are transversally Cantor sets (we refer to \cite{Levitt2}, \cite{BressaudHubertMaass} and \cite{BFG} for the existence of such foliations). One would need to understand how these fit in the geometrical picture to be drawn.

\vspace{3mm}

\paragraph{\bf Invariant measures} One of the key ingredients of the proof of Theorem \ref{maintheorem} is the existence of a measure that satisfies the three following criteria

\begin{enumerate}
\item it is invariant under the action of the Teichmüller flow;
\item it is equivalent to the Lebesgue measure;
\item it relates in a sensible way to the geometry of the underlying dilation surfaces (giving finite mass to subset on which the function $\Theta$ is bounded for instance). 
\end{enumerate}

\noindent The existence of such a measure is in our case a sort of 'low dimensional' miracle. This leads us to ask the following question

\begin{question}
Do strata of moduli spaces of dilation surfaces always have a measure satisfying the three condition above?
\end{question}

\noindent Another range of questions concerns invariant probability measures. In particular it would be interesting to know if there exists an invariant probability measure(for the Teichmüller flow) of maximal complexity \textit{i.e.} whose support contains all minimal points in $\MT$. The same question would also be very interesting in the general setting but seems a bit fanciful as long as the basic topological dynamics of Teichmüller flow are  not understood. 
\begin{question}
Do Patterson-Sullivan-Bowen-Margulis type of measure exist for the action of the Teichmüller flow on $\MT$? What would their dynamical properties be?
\end{question}

\vspace{2mm}

\paragraph{\bf Particular surfaces} A given surface gives rise to  an interesting one-parameter family of foliations.

\begin{question} 
Is it true that for each dilation surface which is not a translation surface the set of Morse-Smale directions in $S^1$ has full measure?
\end{question}

\noindent The fact that dilation tori with two singularities have cylinders is a relatively easy consequence of the existence of convex polygonal models. It does not seem obvious how to generalise this statement to higher genus.

\begin{question} 
Does every dilation surface which is not a translation surface have a dilation cylinder?
\end{question}

\paragraph{\bf Hausdorff dimension}

We have proved that the set of minimal direction has measure zero. What about its Hausdorff dimension? It is easy to prove that every dilation torus has at least one minimal direction and this implies that the Hausdorff dimension of $E$ is at least $\mathrm{dim}(\MT) - 1 = 4$

\begin{question}
Is the Hausdorff dimension of $E$ less than $\mathrm{dim}(\MT) = 5$? Is the Hausdorff dimension of $E$ more than $\mathrm{dim}(\MT) - 1= 4$? 
\end{question}

\noindent Again this question would be interesting in a more general setting but it seems a bit nonsensical to venture to ask questions about the Hausdorff dimension as long as Conjecture \ref{conj} remains open.

\bibliographystyle{alpha} 
\bibliography{biblio}

\end{document}